\documentclass[reqno, 12pt]{amsart}
\pdfoutput=1
\makeatletter
\let\origsection=\section \def\section{\@ifstar{\origsection*}{\mysection}} 
\def\mysection{\@startsection{section}{1}\z@{.7\linespacing\@plus\linespacing}{.5\linespacing}{\normalfont\scshape\centering\S}}
\makeatother        

\usepackage{amsmath,amssymb,amsthm}
\usepackage{mathrsfs}
\usepackage{mathabx}\changenotsign
\usepackage{dsfont}
 
\usepackage{xcolor}
\usepackage[backref]{hyperref}
\hypersetup{
    colorlinks,
    linkcolor={red!60!black},
    citecolor={green!60!black},
    urlcolor={blue!60!black}
}

\usepackage{graphicx}

\usepackage[open,openlevel=2,atend]{bookmark}

\usepackage[abbrev,msc-links,backrefs]{amsrefs} 
\usepackage{doi}

\renewcommand{\PrintDOI}[1]{\doi{#1}}

\usepackage[T1]{fontenc}
\usepackage{lmodern}
\usepackage[babel]{microtype}
\usepackage[english]{babel}

\usepackage{geometry}
\numberwithin{equation}{section}
\numberwithin{figure}{section}

\usepackage{enumitem}

\def\nlabel{\upshape({\itshape \arabic*\,})}

\let\polishlcross=\l
\def\l{\ifmmode\ell\else\polishlcross\fi}

\def\paragraph#1{%
  \noindent\textbf{#1.}\enspace}

\let\setminus=\smallsetminus

\let\sm=\setminus

\makeatletter
\def\moverlay{\mathpalette\mov@rlay}
\def\mov@rlay#1#2{\leavevmode\vtop{   \baselineskip\z@skip \lineskiplimit-\maxdimen
   \ialign{\hfil$\m@th#1##$\hfil\cr#2\crcr}}}
\newcommand{\charfusion}[3][\mathord]{
    #1{\ifx#1\mathop\vphantom{#2}\fi
        \mathpalette\mov@rlay{#2\cr#3}
      }
    \ifx#1\mathop\expandafter\displaylimits\fi}
\makeatother

\DeclareFontFamily{U}  {MnSymbolC}{}
\DeclareSymbolFont{MnSyC}         {U}  {MnSymbolC}{m}{n}
\DeclareFontShape{U}{MnSymbolC}{m}{n}{
    <-6>  MnSymbolC5
   <6-7>  MnSymbolC6
   <7-8>  MnSymbolC7
   <8-9>  MnSymbolC8
   <9-10> MnSymbolC9
  <10-12> MnSymbolC10
  <12->   MnSymbolC12}{}
\DeclareMathSymbol{\powerset}{\mathord}{MnSyC}{180}

\let\epsilon=\varepsilon

\let\rho=\varrho
\let\theta=\vartheta
\let\kappa=\varkappa

\theoremstyle{plain}
\newtheorem{thm}{Theorem}[section]

\newtheorem{prop}[thm]{Proposition}

\newtheorem{lem}[thm]{Lemma}

\theoremstyle{definition}
\newtheorem{remark}[thm]{Remark}
\newtheorem{dfn}[thm]{Definition}
\newtheorem{eg}[thm]{Example}

\usepackage{accents}

\let\phi=\varphi

\DeclareMathOperator{\cf}{cf}
\DeclareMathOperator{\col}{col}
\newcommand{\pomega}{\varpi}

\DeclareSymbolFont{stmry}{U}{stmry}{m}{n}
\DeclareMathSymbol\arrownot\mathrel{stmry}{"58}
\DeclareMathSymbol\Arrownot\mathrel{stmry}{"59}
\def\longarrownot{\mathrel{\mkern5.5mu\arrownot\mkern-5.5mu}}

\def\nlra{\longarrownot\longrightarrow}

\begin{document}

\title[The colouring number of infinite graphs]{The colouring number of infinite graphs}

\author[Nathan Bowler]{Nathan Bowler}
\address{Fachbereich Mathematik, Universit\"at Hamburg,
  Bundesstra\ss{}e~55, D-20146 Hamburg, Germany}
\email{Nathan.Bowler@uni-hamburg.de}

\author[Johannes Carmesin]{Johannes Carmesin}
\address{Department of Pure Mathematics and Mathematical Statistics,
University of Cambridge,
Wilberforce Road, Cambridge CB3 0WB}
\email{j.carmesin@dpmms.cam.ac.uk}

\author[P\'eter Komj\'ath]{P\'eter Komj\'ath}
\address{E\"tv\"s Lor\'and University, H-1053 Budapest, Egyetem t\'er 1-3}
\email{kope@cs.elte.hu}

\author[Christian Reiher]{Christian Reiher}
\address{Fachbereich Mathematik, Universit\"at Hamburg,
  Bundesstra\ss{}e~55, D-20146 Hamburg, Germany}
\email{Christian.Reiher@uni-hamburg.de}

\begin{abstract} 
	We show that, given an infinite cardinal $\mu$, a graph has colouring number at most 
	$\mu$ if and only if it contains neither of two types of subgraph. We also show that 
	every graph with infinite colouring number has a well-ordering of its vertices that 
	simultaneously witnesses its colouring number and its cardinality.
\end{abstract}

\maketitle

\setcounter{footnote}{1}

\section{Introduction}
Our point of departure is a recent article of the third author~\cite{Kom} one of whose results 
addresses infinite graphs with infinite colouring number. Let us recall this notion 
introduced by Erd\H{o}s and Hajnal in~\cite{DefOfColNo}.

\begin{dfn}
\label{dfn:col}
	The {\it colouring number} $\col(G)$ of a graph $G=(V, E)$ is the smallest 
	cardinal~$\kappa$ such that there exists a well-ordering $<^*$ of $V$ with 
	\[
		|N(v)\cap\{w\colon w<^*v\}|<\kappa\qquad\text{for all $v\in V$}\,,
	\]
	where $N(v)$ is the set of neighbours of $v$. We call such well-orderings {\em good}.
\end{dfn}

The result of~\cite{Kom} is that if the colouring number of a graph $G$ is bigger than 
some infinite cardinal $\mu$, then $G$ contains either a $K_\mu$, i.e., $\mu$ mutually 
adjacent vertices, or~$G$ contains for each positive integer $k$ an induced copy of the 
complete bipartite graph $K_{k,k}$. This condition is not a characterisation: there are 
graphs, such as $K_{\omega, \omega}$, which have small colouring number but nevertheless 
include an induced $K_{k,k}$ for each $k$.

Since having colouring number $\le \mu$ is closed not only under taking induced subgraphs 
but even under taking subgraphs, it seems easier to look first for a characterisation in 
terms of forbidden subgraphs. Our main result is that there is indeed a transparent 
characterisation of ``having colouring number $\le\mu$'' in terms of forbidden subgraphs. 
For some explicit graphs called {\it $\mu$-obstructions}, to be introduced in 
Definition~\ref{obstr} below, we shall prove the following.

\begin{thm}\label{col-char}
	Let $G$ be a graph and let $\mu$ denote some infinite cardinal. Then the statement 
	$\col(G)>\mu$ is equivalent to $G$ containing some $\mu$-obstruction as a subgraph.
\end{thm}

This result will also appear in the upcoming book~\cite{KomBook} of the third author.
The proof we describe has an interesting consequence.

\begin{thm}\label{gwo}
	Every graph $G$ whose colouring number is infinite possesses a good well-ordering 
	of length $|V(G)|$.
\end{thm}

It is not hard to re-obtain the result mentioned above from our characterisation,
Theorem~\ref{col-char}, by inspecting whether the $\mu$-obstructions satisfy it. 
In fact, one can easily deduce the following strengthening.

\begin{thm}\label{strong-Komjath}
	If $G$ is a graph with $\col(G)>\mu$, where $\mu$ denotes some infinite cardinal, 
	then~$G$ contains either a $K_\mu$ or, for each positive integer $k$, 
	an induced $K_{k, \omega}$.
\end{thm}

We will also give an example in Section~\ref{sec:nec} demonstrating that the conclusion 
cannot be improved further to the presence of an induced $K_{\omega, \omega}$. Which 
complete bipartite graphs exactly one gets by this approach depends on which properties 
the relevant cardinals have in the partition calculus. 

For standard set-theoretical background we refer to Kunen's textbook~\cite{Kunen}.

\section{Obstructions}\label{sec:obstr}

Throughout this section, we fix an infinite cardinal $\mu$. There are two kinds of 
$\mu$-obstructions relevant for the condition $\col(G)>\mu$ in Theorem~\ref{col-char}. 
They are introduced next.
  
\begin{dfn}\label{obstr}
	(1) A {\it $\mu$-obstruction of type I} is a bipartite graph $H$ with  bipartition $(A, B)$ 
	such that for some cardinal $\lambda\ge\mu$ we have 
	\begin{enumerate}
		\item[$\bullet$] $|A| = \lambda$, $|B| = \lambda^+$,  
		\item[$\bullet$] every vertex of $B$ has at least $\mu$ neighbours in $A$, and
		\item[$\bullet$] every vertex of $A$ has $\lambda^+$ neighbours in $B$.
	\end{enumerate} 

	\smallskip

	(2) Let $\kappa> \mu$ be regular, and let $G$ be a graph with $V(G)=\kappa$. 
	Define $T_G$ to be the set of those $\alpha\in \kappa$ with the following properties:
	\begin{enumerate}
		\item[$\bullet$] $\cf(\alpha)=\cf(\mu)$
		\item[$\bullet$] The order type of $N(\alpha)\cap \alpha$ is $\mu$. 
		\item[$\bullet$] The supremum of $N(\alpha)\cap \alpha$ is $\alpha$. 
	\end{enumerate} 
	If $T_G$ is stationary in $\kappa$, then $G$ is a {\it $\mu$-obstruction of type II}.
	We also call graphs isomorphic to such graphs $\mu$-obstructions of type II. 
\end{dfn}

Now we can directly proceed to the easier direction of Theorem~\ref{col-char}.

\begin{prop}
	If a graph $G$ has a $\mu$-obstruction of either type as a subgraph,
	then ${\col(G)>\mu}$.
\end{prop}

\begin{proof} 
	Suppose first that $G$ contains a $\mu$-obstruction of type $I$, say with 
	bipartition $(A, B)$ as in Definition~\ref{obstr} above, and $|A|=\lambda\ge\mu$. 
	Assume for a contradiction that there is a good well-ordering of $G$.
	Thus every $b\in B$ has a neighbour in $A$ above it in that well-ordering.
	For~$a\in A$, we denote by $X_a$ the set of those neighbours of $a$ that are below $a$ 
	in the well-ordering. Hence $B=\bigcup_{a\in A} X_a$. Since all the $X_a$ have size less 
	than $\mu$, we deduce that~$|B|\leq \lambda$, which is the desired contradiction.

	In the second case, we may without loss of generality assume that $G$ itself is 
	an obstruction of type~II. Again we suppose for a contradiction that there is a good 
	well-ordering $<^*$ of~$V(G)$. 
	Notice that each $\alpha\in T_G$ has a neighbour $\beta<\alpha$ such that $\alpha<^*\beta$. 
	Let  $f \colon T_G \longrightarrow \kappa$ be a function sending each $\alpha$ to some 
	such~$\beta$. By Fodor's Lemma, there must be some $\beta < \kappa$ such that 
	\[
		T = \{\alpha \in T_G \colon f(\alpha) = \beta\}
	\]
	is stationary. 
	Now every element of $T$ is a neighbour of $\beta$, and $\beta$ comes after $T$ 
	in the ordering~$<^*$, which in view of $|T|=\kappa>\mu$ contradicts our assumption 
	that this ordering is good.
\end{proof} 

We say that a graph is {\it $\mu$-unobstructed} if it has no $\mu$-obstruction of either 
type as a subgraph. To complete the proof of Theorem~\ref{col-char} we still need to show 
that every $\mu$-unobstructed graph~$G$ satisfies $\col(G)\le\mu$. This will be the objective 
of Sections~\ref{sec:reg} and~\ref{sec:sing}.

In the remainder of this section, we prove two results asserting that in order to find an 
obstruction in a given graph $G$ it suffices to find something weaker.

\begin{dfn} 
	A {\it $\mu$-barricade} is bipartite graph with bipartition $(A, B)$ such that  
	\begin{enumerate}
	\item[$\bullet$] $|A|<|B|$, 
	\item[$\bullet$] and every vertex of $B$ has at least $\mu$ neighbours in $A$.
	\end{enumerate} 
\end{dfn}
 
\begin{lem}\label{barricade}
	If $G$ has a $\mu$-barricade as a subgraph, then it also has a $\mu$-obstruction 
	of type~I as a subgraph.
\end{lem}

\begin{proof}
	Let $H$ with bipartition $(A, B)$ be a barricade which is a subgraph of $G$, 
	chosen so that $\lambda = |A|$ is minimal. 
	By deleting some vertices of $B$ if necessary, we may assume that~$B$ has cardinality 
	$\lambda^+$. 
	Let $A'$ be the set of $a \in A$ for which $N_{B}(a)$ is of size $\lambda^+$, 
	and let~$B'$ be the set of elements of $B$ with no neighbour in $A \setminus A'$. 
	By the definition of $A'$, there are at most $\lambda$ edges $ab$ with 
	$a \in A \setminus A'$ and $b \in B$. 
	So $B \setminus B'$ is of size at most $\lambda$.  
	It follows that~$B'$ has cardinality~$\lambda^+$. 
	In particular, the subgraph $H'$ of $H$ on $(A', B')$ is a barricade, so by minimality of 
	$|A|$ we have $|A'| = \lambda$. 
	Since by construction every vertex of $A'$ has $\lambda^+$ neighbours in $B$ and hence 
	in $B'$, the subgraph $H'$ is a $\mu$-obstruction of type~I.
\end{proof}

\begin{dfn}
	Let $\kappa>\mu$ be regular. A graph $G$ with set of vertices $\kappa$ is said to be a 
	{\it $\mu$-ladder} if there is a stationary set $T$ such that each $\alpha\in T$ has 
	at least $\mu$ neighbours in $\alpha$. Also, every graph isomorphic to such a graph 
	is regarded as a $\mu$-ladder.
\end{dfn}

\begin{lem}\label{ladder-enough}   
	Every graph containing a $\mu$-ladder is $\mu$-obstructed.
\end{lem}

\begin{proof}
	It suffices to prove that every $\mu$-ladder is $\mu$-obstructed. So let $G$ with 
	$V(G)=\kappa$ and the stationary set $T$ be as described in the previous definition. 
	For each $\alpha\in T$ we let the sequence $\langle \alpha_i\,|\,i<\mu\rangle$ enumerate 
	the $\mu$ smallest neighbours of 
	$\alpha$ in increasing order and denote the limit point of this sequence by $f(\alpha)$. 
	Clearly we have $f(\alpha)\le\alpha$ and $\cf\bigl(f(\alpha)\bigr)=\cf(\mu)$ for all 
	$\alpha\in T$. 

	Let us first suppose that the set 
	\[
		T'=\{\alpha\in T \colon f(\alpha)<\alpha\}
	\]
	is stationary in $\kappa$. Then for some $\gamma<\kappa$ the set
	\[
		B=\{\alpha\in T'\colon f(\alpha)=\gamma\}
	\]
	is stationary and as $|\gamma|<\kappa=|B|$ the pair $(\gamma, B)$ is a $\mu$-barricade 
	in $G$. Due to Lemma~\ref{barricade} it follows that $G$ contains a $\mu$-obstruction 
	of type~$I$. 

	So it remains to consider the case that 
	\[
		T''=\{\alpha\in T\colon f(\alpha)=\alpha\}
	\]
	is stationary in $\kappa$. In that case we have 
	$N(\alpha)\cap\alpha=\{\alpha_i \colon i<\mu\}$ 
	for all $\alpha\in T''$. So $T_G$ is a superset of $T''$ 
	and thus stationary, meaning that $G$ is a $\mu$-obstruction of type~II.
\end{proof}

\section{Regular \texorpdfstring{$\kappa$}{kappa}}
\label{sec:reg}
In this and the next section we shall prove the harder part of Theorem~\ref{col-char}, 
in such a way that Theorem~\ref{gwo} is also immediate. To this end we shall show

\begin{thm}\label{dicho}
	Let $G$ denote an infinite graph of order $\kappa$ and let $\mu$ be an infinite cardinal.
	Then at least one of the following three cases occurs:
	\begin{enumerate}
		\item[$\bullet$] $G$ has a subgraph $H$ with $|V(H)| < |V(G)|$ and $\col(H)>\mu$.
		\item[$\bullet$] $G$ is $\mu$-obstructed.
		\item[$\bullet$] $G$ has a good well-ordering of length $\kappa$ exemplifying 
			$\col(G)\le\mu$.
	\end{enumerate}
\end{thm}

Suppose for a moment that we know this. To deduce Theorem~\ref{col-char} 
we consider any graph with $\col(G)>\mu$. 
Let $G^*$ be subgraph of $G$ with $\col(G^*)>\mu$ and subject to this 
with $|V(G^*)|$ as small as possible. 
Then $G^*$ is still infinite and when we apply Theorem~\ref{dicho} to~$G^*$ the first and 
third outcome are impossible, so the second one most occur. Thus~$G^*$ and hence~$G$ 
contains a $\mu$-obstruction, as desired. 
To obtain Theorem~\ref{gwo} we apply Theorem~\ref{dicho} to $G$ with~$\mu=\col(G)$.

The proof of Theorem~\ref{dicho} itself is divided into two cases according to whether 
$\kappa$ is regular or singular. The former case will be treated immediately and the latter 
case is deferred to the next section.

\begin{proof}[Proof of Theorem~\ref{dicho} when $\kappa$ is regular] 
	Let $V(G)=\kappa$ and consider the set
	\[
		T=\{\alpha<\kappa\colon \text{some $\beta\ge\alpha$ has at least $\mu$ 
			neighbours in $\alpha$}\}\,.
	\]

	\smallskip

	{\it \hskip1cm First Case: $T$ is not stationary in $\kappa$.}
	
	\smallskip

	We observe that $0\not\in T$. 
	Let $\langle \delta_i\,|\,i<\kappa\rangle$ be a strictly increasing continuous sequence 
	of ordinals with limit $\kappa$ starting with $\delta_0=0$ and such that~$\delta_i\not\in T$ 
	holds for all $i<\kappa$. 
	Now if for some~$i<\kappa$ the restriction $G_i$ of $G$ to the half-open interval 
	$[\delta_i, \delta_{i+1})$ has colouring number~$>\mu$, then the first alternative holds. 
	Otherwise we may fix for each $i<\kappa$ a well-ordering $<_i$ of~$V(G_i)$ that exemplifies 
	$\col(G_i)\le\mu$. The concatenation $<^*$ of all these well-orderings has length $\kappa$, 
	so it suffices to verify that it demonstrates $\col(G)\le\mu$. 

	To this end, we consider any vertex $x$ of $G$. Let $i<\kappa$ be the ordinal with 
	$x\in G_i$. The neighbours of $x$ preceding it in the sense of $<^*$ are either in 
	$\delta_i$ or they belong to $G_i$ and precede $x$ under $<_i$. 
	Since $x\ge\delta_i$ and $\delta_i\not\in T$, there are less than $\mu$ neighbours of $x$ 
	in $\delta_i$. Also, by our choice of $<_i$, there are less than $\mu$ such neighbours 
	in $G_i$. 

	\smallskip

	{\it \hskip1cm Second Case: $T$ is stationary in $\kappa$.}

	\smallskip

	Let us fix for each $\alpha\in T$ an ordinal $\beta_\alpha\ge\alpha$ with 
	$|N(\beta_\alpha)\cap\alpha|\ge\mu$. A standard argument shows that the set
	\[
		E=\{\delta<\kappa\colon\text{if $\alpha\in T\cap\delta$, then $\beta_\alpha<\delta$}\}
	\]
	is club in $\kappa$. Thus $T\cap E$ is unbounded in $\kappa$. Let the sequence 
	$\langle \eta_i\,|\,i<\kappa\rangle$ enumerate the members of this set in increasing order. 
	Then for each $i<\kappa$ the ordinal $\xi_i=\beta_{\eta_i}$ is at least~$\eta_i$ and smaller 
	than $\eta_{i+1}$, because the latter ordinal belongs to $E$. In particular, each of the 
	equations $\eta_i=\xi_j$ and $\xi_i=\xi_j$ is only possible if $i=j$. 
	Thus it makes sense to define
	\[
		v_\alpha=\begin{cases}
			\alpha &\text{ if } \alpha\ne \eta_i, \xi_i \text{ for all } i<\kappa\,,\cr
			\xi_i  &\text{ if } \alpha= \eta_i \text{ for some } i<\kappa\,,\cr
			\eta_i &\text{ if } \alpha= \xi_i \text{ for some } i<\kappa\,.\cr
		\end{cases}
	\]
	The map $\pi$ sending each $\alpha<\kappa$ to $v_\alpha$ is a permutation of $\kappa$. 
	If $\alpha$ belongs to the stationary set $T\cap E$, then $v_\alpha=\xi_i$ for some 
	$i<\kappa$ and therefore $v_\alpha$ has at least $\mu$ neighbours in $\eta_i$ and all 
	of these are of the form $v_\beta$ with $\beta<\alpha$. So $\pi$ gives an isomorphism 
	between $G$ and a $\mu$-ladder, and in the light of Lemma~\ref{ladder-enough} we are done.
\end{proof}

\section{Singular \texorpdfstring{$\kappa$}{kappa}}
\label{sec:sing}

Next we consider the case that $\kappa$ is a singular cardinal. 
The form of our argument will be recognisable to anyone who is familiar with Shelah's 
singular compactness theorem (see for instance~\cite{Sh:52}). We will not, however, 
assume such familiarity. 

Throughout this section, sets of size at 
least $\mu$ will be referred to as {\em big} and sets 
of size less than~$\mu$ will be said to be {\em small}.
We will often consider $\subseteq$-increasing sequences ${\langle X_i\,|\,i < \gamma\rangle}$ 
of sets for which each $N_{X_i}(v)$ is small. In such cases we would like to conclude that 
also~$N_{\bigcup_{i < \gamma}X_i}(v)$ is small. 
We can do this as long as $\gamma$ and $\mu$ have 
different cofinalities. So we fix the notation~$\pomega$ for the rest of the argument to mean 
the least infinite cardinal whose cofinality is not equal to $\cf(\mu)$. Thus $\pomega$ is 
either $\omega$ or $\omega_1$.

\begin{dfn}
	A set $X$ of vertices of a graph $G$ is {\em robust} if for any $v \in V(G) \setminus X$ 
	the neighbourhood $N_X(v)$ is small.
\end{dfn}

\begin{remark}\label{pomegaseq}
	Let $\langle X_i\,|\,i < \pomega\rangle$ be a $\subseteq$-increasing sequence of robust sets. 
	Then $\bigcup_{i < \pomega} X_i$ is also robust.
\end{remark}

\begin{lem}\label{extlem}
	Let $G$ be a $\mu$-unobstructed graph and let $X$ be an uncountable set of vertices of $G$. 
	Then there is a robust set $Y$ of vertices of $G$ which includes $X$ and is of the same 
	cardinality.
\end{lem}

\begin{proof}
	Let $\lambda$ be the cardinality of $X$. We build a $\subseteq$-increasing sequence 
	$\langle X_i\,|\,i < \pomega\rangle$ of sets recursively by letting $X_0 = X$, 
	taking $X_{i+1} = X_i \cup \{v \in V(G)\colon N_{X_i}(v) \text{ is big}\}$ in the 
	successor step and $X_\ell = \bigcup_{i < \ell} X_i$ for $\ell$ a limit ordinal. 
	Finally we set $Y = \bigcup_{i < \pomega} X_i$. 
	Since by construction $Y$ is robust and includes $X$, it remains to prove that 
	$|Y| = \lambda$. 

	To do this, we prove by induction on $i$ that each $X_i$ is of size $\lambda$. 
	The cases where $i$ is 0 or a limit are clear, so suppose $i = j + 1$. 
	By the induction hypothesis, $|X_j| = \lambda$. If $|X_{j+1}|$ were greater than $\lambda$, 
	then the induced bipartite subgraph on $(X_j, X_{j+1}\sm X_j)$ would be a $\mu$-barricade, 
	which is impossible by Lemma~\ref{barricade}. Thus $|X_{j+1}| = \lambda$, as required.
\end{proof}

\begin{remark}
	Lemma~\ref{extlem} also holds when $X$ is countably infinite, but the proof is more 
	involved and so we have omitted it (unlike in the above proof, we need that there 
	are no type~II obstructions).
\end{remark}

\begin{proof}[Proof of Theorem~\ref{dicho} when $\kappa$ is singular] 
	If $G$ is $\mu$-obstructed then we are done, so we suppose that it is not.
	Let us fix any bijective enumeration $\langle v_i\,|\,i < \kappa\rangle$ of the set of 
	vertices and a continuous increasing sequence $\langle \kappa_i\,|\,i < \cf(\kappa)\rangle$
	of cardinals with limit $\kappa$, where $\kappa_0 > \cf(\kappa)$ is uncountable. 

	We begin by building a family $\langle X_{i,j}\,|\,i < \cf(\kappa), j<\pomega\rangle$ of 
	robust sets of vertices of $G$, with~$X_{i,j}$ of size $\kappa_i$. This will be done 
	by nested recursion on $i$ and $j$. When we come to choose $X_{i,j}$, we will already have 
	chosen all $X_{i',j'}$ with $j' < j$ or with both $j' = j$ and $i' < i$. 
	Whenever we have just selected such a set $X_{i,j}$, we fix immediately an arbitrary 
	enumeration $\langle x_{i,j}^k\,|\,k < \kappa_i\rangle$ of this set. We impose the 
	following conditions on this construction:

	\begin{enumerate}[label=\nlabel]
		\item\label{it:1} $\{v_k\colon k < \kappa_{i}\}\subseteq X_{i,0}$ for 
			all $i < \cf(\kappa)$.
		\item\label{it:2} $\bigcup_{i' \leq i, j' \leq j} X_{i',j'} \subseteq X_{i,j}$ for 
			all $i < \cf(\kappa)$ and $j<\pomega$.
		\item\label{it:3} $\{x_{i', j}^k\colon k < \kappa_i\}\subseteq X_{i,j+1}$ for all 
			$i< i' < \cf(\kappa)$ and $j<\pomega$.
	\end{enumerate}

	These three conditions specify some collection of $\kappa_i$-many vertices which must appear 
	in~$X_{i,j}$. By Lemma~\ref{extlem} we can extend this collection to a robust set of the same 
	size and we take such a set as $X_{i,j}$. This completes the description of our recursive 
	construction.
 
	The purpose of condition~\ref{it:3} is to ensure that we have
	\begin{enumerate}[label=\nlabel, resume]
		\item\label{it:4} $X_{\ell, j}\subseteq \bigcup_{i<\ell} X_{i, j+1}$ 
			whenever $\ell<\cf(\kappa)$ 
			is a limit ordinal and $j<\pomega$.
	\end{enumerate}

	Indeed, for any $x \in X_{\ell, j}$ there is some index $k<\kappa_\ell$ with 
	$x=x_{\ell, j}^k$, owing to the continuity of the $\kappa_i$ there is some 
	ordinal $i < \ell$ with $k < \kappa_i$, and condition~\ref{it:3} yields 
	$x\in X_{i, j+1}$ for any such $i$. 

	Now for $i < \cf(\kappa)$ the set $X_i = \bigcup_{j < \pomega}X_{i,j}$ is robust 
	by Remark~\ref{pomegaseq}. We claim that for any limit ordinal $\ell<\cf(\kappa)$ 
	we have $X_\ell = \bigcup_{i < \ell}X_i$. 
	That each $X_i$ with $i < \ell$ is a subset of~$X_\ell$ is clear by 
	condition~\ref{it:2} above. 
	The other inclusion follows by taking the union over all $j<\pomega$ in~\ref{it:4}.
 
	Each vertex must lie in some set $X_i$ by condition~\ref{it:1} above, and it follows 
	from what we have just shown that the least such $i$ can never be a limit. 
	That is, $X_0$ together with all the sets~$X_{i+1} \setminus X_i$ gives a partition 
	of the vertex set. If the induced subgraph of $G$ on any of these sets has colouring 
	number $>\mu$, then the first alternative of Theorem~\ref{dicho} holds. 
	Otherwise all of these induced subgraphs have good well-orderings. 
	Since each $X_i$ is robust, the well-ordering obtained by concatenating all of these 
	well-orderings is also good, so that the third alternative of Theorem~\ref{dicho} holds.
\end{proof}

\section{A necessary condition}\label{sec:nec}

In this section we show that we can now easily deduce Theorem~\ref{strong-Komjath}. 
We shall rely on the following result of Dushnik, Erd\H{o}s, and Miller from~\cite{DEM}.

\begin{thm}
	For each infinite cardinal $\lambda$ we have 
	$\lambda\longrightarrow(\lambda, \omega)$. 
	This means that if the edges of a complete graph on $\lambda$ vertices are coloured 
	red and green, then there is either a red clique of order $\lambda$, or a green clique 
	of order $\omega$.
\end{thm}    

By restricting the attention to the red graph, one realises that this means that every 
infinite graph $G$ either contains a clique of order $|V(G)|$ or an infinite independent set. 
When used in this formulation, we refer to the above theorem as DEM.

\begin{proof}[Proof of Theorem~\ref{strong-Komjath}]\label{Kom}
	By Theorem~\ref{col-char} it remains to show that every graph with an obstruction 
	of type~I or~II has a $K_\mu$ subgraph or an induced $K_{k,\omega}$.

	First we check this for obstructions $(A,B)$ of type~I.
	By DEM, we may assume that the neighbourhood $N(b)$ of every $b\in B$ contains an 
	independent set $Y_b$ of size $k$.
	Let~$f$ be the function mapping $b$ to $Y_b$.
	There must be a $k$-element subset $Y$ of $A$ such that~${|f^{-1}[Y]|=|B|}$.
	By~DEM again, we may assume that $f^{-1}[Y]$ contains an infinite independent set $B'$.
	Then~$G[B'\cup Y]$ is isomorphic to  $K_{k,\omega}$.

	Hence it remains to show that every obstruction $G$ of type II has a~$K_\mu$ subgraph 
	or an induced~$K_{k,\omega}$.  
	For every $\alpha\in T_G$, we may assume by DEM that $N(\alpha)\cap\alpha$ contains an 
	independent set $Y_\alpha$ of size $k$.
	For each $i$ with $1\leq i\leq k$, let $f_i\colon T \to \kappa$ be the function mapping 
	$\alpha$ to the $i$-th smallest element of $Y_\alpha$.
	By Fodor's Lemma, there is some stationary $T'\subseteq T_G$ at which $f_1$ is constant, 
	and some stationary $T''\subseteq T'$ at which $f_2$ is constant.
	Proceeding like this, we find some stationary $S\subseteq T_G$ at which all the $f_i$ are 
	constant. 
	Let~$X$ be the set of these $k$ constants.
	By DEM, we may assume that $S$ contains a countably infinite independent set $I$.
	Then $G[X\cup I]$ is isomorphic to  $K_{k,\omega}$.
\end{proof}

In the following example, we show that if we replace `$K_{k,\omega}$' 
by `$K_{\omega,\omega}$' in Theorem~\ref{strong-Komjath}, then it becomes false.

\begin{eg}\label{no_KOmegaOmega}
	Let $A$ be the set of finite 0-1-sequences, and let $B$ be the set of 0-1-sequences
	with length~$\omega$.
	We define a bipartite graph $G$ with vertex set $A\cup B$ by adding for each 
	$a\in A$ and $b\in B$ the edge $ab$ if $a$ is an initial segment of $b$. 
	Since $G$ is bipartite, it cannot contain a $K_\omega$. It cannot contain a 
	$K_{\omega,\omega}$ either, since any two vertices in $B$ have only finitely 
	many neighbours in common. 
	On the other hand, $\col(G)>\aleph_0$, since $G$ is an $\aleph_0$-barricade. 
\end{eg}

\begin{remark}
	The proof of Theorem~\ref{strong-Komjath} actually shows something slightly stronger: 
	in order to have $\col(G)\leq \mu$ it is enough to forbid $K_\mu$ and a 
	$K_{k,\mu^+}$-subgraph where the $k$ vertices on the left are independent. 
	If $\mu=\omega$, then DEM implies it is enough to forbid $K_\mu$ and an \emph{induced} 
	$K_{k,\mu^+}$. 
	On the other hand if $\kappa=2^\omega$ and $\mu=\omega_1$, it may happen that the 
	bipartite graph contains neither a $K_\mu$ nor an induced $K_{k,\omega_1}$ 
	by Sierpi\'nski's theorem from~\cite{Sierp33}, which says that
	\[
		2^{\omega}\nlra (\omega_1)^2_2\,.
	\]
\end{remark}

Our characterisation simplifies the study of many questions about colouring numbers, 
since they can often be reduced to questions about the properties of our obstructions. 
However there are some cases where our results do not appear to be helpful. 
For example, Halin showed in~\cite{Halinsbook} that if $\lambda$ is infinite and a graph 
$G$ has colouring number greater than $\lambda$, then $G$ includes a subdivision of 
$K_{\lambda}$. But this result is more closely tied to the structure of graphs with 
no subdivision of $K_{\lambda}$ than of those with colouring number less than $\lambda$, 
and our methods appear not to provide a simplification of the proof.

\end{document}